\documentclass[a4paper,12pt]{article}

\usepackage{amsmath,amssymb,amstext,amsfonts,latexsym,amsthm}
\usepackage{color}
\usepackage{dsfont,mathptmx}
\usepackage{graphicx}
\usepackage{hyperref}
\usepackage{enumerate}

\newtheorem{theorem}{Theorem}
\newtheorem{lemma}{Lemma}[section]
\newtheorem{proposition}{Proposition}[section]
\newtheorem{definition}{Definition}
\newtheorem{corollary}{Corollary}[theorem]

\newtheorem*{example}{Examples}
\newtheorem*{MarkovTh}{Markov's Theorem}

\newcommand{\CC}{\mathds{C}}

\newcommand{\ZZp}{\mathds{Z}_{+}}
\newcommand{\NN}{\mathds{N}}
\newcommand{\HH}{\mathds{H}}

\newcommand{\PP}{\mathds{P}}

\newcommand{\RR}{\mathds{R}}

\newcommand{\ZZ}{\mathds{Z}}

\newcommand{\ch}[1]{\mathbf{C}_h\left(#1 \right)}

\newcommand{\supp}[1]{\mathrm{supp}\left(#1 \right)}

\newcommand{\dgr}[1]{\deg(#1)}
\newcommand{\dsty}{\displaystyle}

\newcommand{\unifn}{\;{\mathop{\rightrightarrows}_{n}}\;}
\newcommand{\limn}{\;{\mathop{\longrightarrow}_{n \to \infty}}\;}

\newcommand{\NC}{\mathbf{M}(0,1)}
\newcommand{\Nzero}{{\mathcal{N}}_{0}}
\newcommand{\Nsign}{{\mathcal{N}}_{1}}

\title{Rational Approximation and  Sobolev-type  Orthogonality}

\author{Abel D\'{\i}az-Gonz\'{a}lez\thanks{Supported by the Research Fellowship Program, Ministerio de Econom\'{\i}a y Competitividad of Spain,  under grant  MTM2015-65888-C4-2-P.}\\Universidad Carlos III de Madrid\\ \emph{abdiazgo@math.uc3m.es} \and
H\'{e}ctor Pijeira-Cabrera\thanks{Research partially supported by  Ministry of Science, Innovation and Universities of Spain, under grant  PGC2018-096504-B-C33} \\ Universidad Carlos III de Madrid \\  \emph{hpijeira@math.uc3m.es}
\and Ignacio P\'{e}rez-Yzquierdo\thanks{Research partially supported by  Fondo Nacional de Innovaci\'{o}n  y Desarrollo Cient\'{\i}fico y Tecnol\'{o}gico (FONDOCYT),  Dominican Republic, under grant   {2015-1D2-164}.} \\ Universidad Aut\'{o}noma de Santo Domingo \\ \emph{igca58@gmail.com}}

\begin{document}
\maketitle

\begin{abstract}
In this paper, we study  the sequence of orthogonal polynomials $\{S_n\}_{n=0}^{\infty}$ with respect to the   Sobolev-type  inner product
 $$\langle f,g \rangle=  \int_{-1}^{1} f(x) g(x) \,d\mu(x) +\sum_{j=1}^{N} \eta_{j} \,f^{(d_j)}(c_{j}) g^{(d_j)}(c_{j}),
$$
where $\mu$  is in the Nevai class $\NC$,  $\eta_j >0$, $N,d_j \in \ZZp$  and $\{c_1,\dots,c_N\}\subset \RR \setminus [-1,1]$.   Under some restriction of order in the discrete part of $\langle\cdot,\cdot \rangle$, we prove that for sufficiently large $n$ the zeros of $S_n$ are real, simple, $n-N$ of them   lie  on $(-1,1)$  and each of the  mass points $c_j$ ``attracts''  one of the remaining $N$ zeros.

The sequences of associated polynomials $\{S_n^{[k]}\}_{n=0}^{\infty}$ are defined for each  $k\in \ZZp$.  We prove an analogous of   Markov's Theorem on rational approximation to a function of certain class of holomorphic functions and  we give an estimate of the ``speed''  of convergence.
\end{abstract}

%


\section{Introduction}\label{Sec-Intro}

Let $\mu$ be a finite positive Borel measure whose support $\supp{\mu} \subset [-1,1]$ contains an infinite set of points,   and  $\{P_{n}\}_{n\geq 0}$ be  the sequence of  monic orthogonal polynomials with respect to $\mu$, defined by the  relations
\begin{equation}\label{Standard-OP}
\langle x^k, P_n\rangle_{\mu} =\int_{-1}^{1} x^k\; P_n(x)\, d\mu(x)= 0, \qquad  k=0,1,\dots,(n-1).
\end{equation}
These polynomials satisfy the  three-term recurrence relation
\begin{align}\label{Standard-3TRR}
P_{n+1}(z)=& (z-b_n)P_{n}(z)-a_n^2  P_{n-1}(z), \quad n \geq  0,\\ \nonumber
& P_{-1}(z)=0 \quad \text{and} \quad P_0(z)=1;
\end{align}
where  $a_0\neq 0$ is an arbitrary constant, $a_n={\|P_{n}\|_{\mu}}/{\|P_{n-1}\|_{\mu}}$ for $n>0$, $b_n={\langle P_n, x\; P_n\rangle_{\mu} }/{\|P_{n}\|^2_{\mu}}$ and  $\|\cdot\|_{\mu}=\sqrt{\langle \cdot, \cdot \rangle_{\mu} }$. Usually, an inner product is called standard if the multiplication operator is symmetric with respect to the inner product, i.e., $\langle x f,g\rangle_{\mu}=\langle  f,x g\rangle_{\mu}$. Clearly, \eqref{Standard-OP} is standard and \eqref{Standard-3TRR} is an immediate consequence of \eqref{Standard-OP} , which turns out to be an essential tool in the theory of  standard orthogonal polynomials.

We say that a measure $\mu$  with  support  $[-1,1]$  is in the Nevai class $\NC$, $\mu \in \NC$, if the corresponding  sequence of orthogonal polynomials  $\{P_{n}\}_{n\geq 0}$ satisfies the recurrence relation \eqref{Standard-3TRR},  when  $\dsty \lim_{n \to \infty}a_n=1/2$ and $\dsty \lim_{n \to \infty}b_n=0$. The condition $\mu^{\prime}>0$ a.e. on $[-1,1]$  is a sufficient condition for $\mu \in \NC$ (c.f. \cite{Rakh77,Rakh82}). The class $\NC$ has been thoroughly studied in \cite{Nev79}, where it is proved that $\mu \in \NC$ is equivalent to
\begin{equation}\label{AympRatioStandard}
\frac{P_{n+1}(z)}{P_{n}(z)} \unifn \frac{\varphi(z)}{2}, \quad K \subset  \Omega=\CC \setminus [-1,1],
\end{equation}
where  $\varphi(z) = z +\sqrt{z^2-1}$   ($\sqrt{z^2-1}>0$ for $z>1$) is the function which maps the complement of $[-1,1]$ onto the exterior of the unit circle. Throughout this paper, we use the notation $\dsty f_n \unifn f; \;K \subset U$ when the sequence of functions $f_n$ converges to $f$ uniformly on every compact subset $K$ of the region $U$.

Let us denote by  $P_n^{[1]}$ the usually called \emph{$n$th polynomial  associated to $P_n$}, defined by the expression
  $$P_n^{[1]}(z)= \int_{-1}^{1} \frac{P_{n+1}(z)-P_{n+1}(x)}{z-x}d\mu(x).$$
Note that $P_n^{[1]}$ is a polynomial of degree $n$  with leading coefficient equal to $\mu([-1,1])$, which satisfies the  three-term recurrence relation
\begin{align}\label{Standard-3TRR-2}
P_{n+1}^{[1]}(z)=& (z-b_{n+1})P_{n}^{[1]}(z)-a_{n+1}^2  P_{n-1}^{[1]}(z), \quad n \geq  0,\\ \nonumber
& P_{-1}^{[1]}(z)=0 \quad \text{and} \quad P_0^{[1]}(z)=\mu([-1,1]).
\end{align}

As it  is known, some particular families of orthogonal polynomials were studied in detail before a general theory existed. One of the starting points of this theory is closely related to  the study of the convergence of certain sequences of rational functions, as  can be seen  in the first  treatises on the subject \cite[Ch. I-\S 4,]{Sho34} and \cite[\S 3,5]{Szg75}. The analysis of the convergence of these sequences   entails essential difficulties. One of the first, and most remarkable, general results in this sense  is the following theorem established by A. A. Markov in 1895.

\begin{MarkovTh}[{\cite[Th. 6.1]{NikSor91}}]\label{MarkovTh} Let $\mu$ be a finite positive Borel measure supported in   $[-1,1]$. Then
$$\frac{P_n^{[1]}(z)}{P_{n+1}(z)} \unifn \hat{\mu}(z), \quad K \subset  \Omega_\infty=\overline{\CC} \setminus [-1,1],
$$
where $\dsty \hat{\mu}(z)=\int_{-1}^{1}\frac{d\mu(x)}{z-x}$  is known as  Markov's function of $\mu$.
\end{MarkovTh}
Note that $\hat{\mu}(z)$ is well defined and holomorphic in $\Omega_{\infty}$ ( $\hat{\mu} \in \HH(\Omega_{\infty})$ for short). Some examples can be seen in
\cite[p. 64]{NikSor91}.  This classical theorem admits several generalizations,  some of which are discussed in  \cite{BerLago01,GonLagRak89,GonRakSor97,Lago85}  and references therein.

We define the \emph{discrete  Sobolev inner product }  through the expression
\begin{align}\label{IP-Sobolev}
 \langle f,g \rangle= & \int_{-1}^{1} f(x) g(x) \,d\mu(x) +\sum_{j=1}^{N} \sum_{i=0}^{d_j  } \eta_{j,i} \,f^{(i)}(c_{j}) g^{(i)}(c_{j});
\end{align}
where $\mu$ is as above, $N\geq 0$,  $\eta_{j,i}\geq 0$,   $\eta_{j,d_j}> 0$, $c_j \in \RR\setminus [-1,1]$,  $d_j  \in \ZZ_+$ and  $f^{(i)}$ denotes the $i$th derivative of a function $f$.

For $n \in \ZZp$ we denote by $S_n$ the monic polynomial of lowest degree  satisfying
\begin{equation}\label{Sobolev-Orth}
\langle x^k,S_n \rangle = 0, \quad  \text{for } \; k=0,1,\dots,n-1.
\end{equation}
It is easy to see that for every  $n\in \ZZp$, there exists a unique polynomial $S_n$ of degree $n$. In fact, the existence of such polynomials  is deduced by solving a homogeneous linear system with $n$ equations and $n+1$ unknowns. Uniqueness follows from the minimality of the degree for the polynomial solution.

We refer the reader to \cite{MaXu15,And01}  for a review of  this type of non-standard  orthogonality. As is well known, most arguments for the standard theory of orthogonal polynomials  fail in the Sobolev  case. As   shown in the next examples,  it is no longer true that the zeros lie on the convex hull of the support of the measures involved in the inner product.
\begin{example}\
\begin{enumerate}
  \item  Set $\dsty  \langle f,g \rangle=  \int_{-1}^{1} f(x) g(x) \,dx+f^{\prime}(3) g^{\prime}(3)+f^{\prime\prime}(2) g^{\prime\prime}(2)$, then
$$
S_5(x)=x^5 + \frac{11282625}{1995289} x^4+\frac{202236410}{1795760}  x^3+\frac{28506900}{1995289}  x^2-\frac{438413755}{41901069} x -\frac{11758825}{1995289},
$$
whose zeros are approximately $\xi_{1}\approx 0.4 $, $\xi_{2}\approx- 0.7$, $\xi_{3}\approx 1.1+2i$, $\xi_{4}\approx  1.1-2i$ and $\xi_{5}\approx 3.8 $. Note that three of them are out of $[-1,1]$ and two are not real numbers.

\item  Set $\dsty  \langle f,g \rangle=  \int_{-1}^{1} f(x) g(x)(1-x) \,dx+f^{\prime}(3) g^{\prime}(3)+f^{\prime\prime}(2) g^{\prime\prime}(2)$, then
  $$
S_5(x)=x^5+\frac{57943145}{27312164}x^4-\frac{242237045}{13656082}x^3-\frac{522277585}{20484123}x^2-\frac{53214815}{40968246}x+\frac{220912645}{52141404},
$$ whose zeros are approximately $\xi_{1}\approx 0.3 $, $\xi_{2}\approx- 0.6$, $\xi_{3}\approx -1.1$, $\xi_{4}\approx  3.9$ and $\xi_{5}\approx -4.7 $. Note that three zeros are out of $[-1,1]$ and two of them,  escape to the opposite side  where  the mass  points  are found.
\end{enumerate}
\end{example}

\begin{definition}\label{Set-SOrdered-General}
Let $\{(r_j,\nu_j)\}_{j=1}^M\!\subset \!\RR\!\times\!\ZZp$ be a finite sequence of $M$ ordered pairs and $A\subset \RR$. We say that $\{(r_j,\nu_j)\}_{j=1}^M $ is \emph{sequentially-ordered with respect to $A$}, if
\begin{enumerate}
\item $0\leq \nu_1\leq \nu_2\leq \cdots \leq \nu_M$.
\item $r_k\notin\ch{A \cup\{r_1,r_2,\dots,r_{k-1}\}}\dsty$ for $k=1,2,\dots,M$; where $\ch{B}$ denotes the convex hull of an arbitrary set $B\subset \CC$.
\end{enumerate}
If $A=\emptyset$, we say that  $\{(r_j,\nu_j)\}_{j=1}^M $ is \emph{sequentially-ordered}  for brevity.

We say that the discrete Sobolev inner product \eqref{IP-Sobolev} is \emph{sequentially-ordered}, if the set of ordered pairs $\{(c_j,i): 1\leq j\leq N, 0\leq i\leq d_j \text{ and } \eta_{j,i}>0 \}$ may be arranged to form a finite sequence of ordered pairs which is sequentially ordered with respect to $(-1,1)$.
\end{definition}

 From the second condition of Definition \ref{Set-SOrdered-General}, the coefficient $\eta_{j,d_j}$ is the only coefficient $\eta_{j,i}$ ($i=0,1,\dots,d_j$) different from zero,  for each $j=1,2,\dots,N$. Hence,  \eqref{IP-Sobolev} takes the form
\begin{align}\label{IP-Sobolev-SO}
\langle f,g \rangle=  \int_{-1}^{1} f(x) g(x) \,d\mu(x) +\sum_{j=1}^{N} \eta_{j,d_j} \,f^{(d_j)}(c_{j}) g^{(d_j)}(c_{j}).
\end{align}

Note that the inner products involved in the previous examples are not sequentially-ordered.  In most of our work, we will restrict our attention to sequentially-ordered discrete Sobolev inner products. The following theorem shows our reasons for this assumption.

\begin{theorem}\label{Th_ZerosSimp}  If \eqref{IP-Sobolev-SO}   is  a sequentially-ordered discrete Sobolev inner product, then  $S_n$ has at least $n-N$ changes of sign on $(-1,1)$.
\end{theorem}

 The  previous Theorem  is  still true if  $c_j=-1$ or $c_j=1$, for some $j$. Furthermore, if $N=1$ in \eqref{IP-Sobolev-SO}, from  Theorem \ref{Th_ZerosSimp}  we get that all the zeros of  $S_n$  are real, simple, and at most one of them is outside of $(-1,1)$.

If $n\leq N$,  $S_n$ can have changes of sign on $(-1,1)$ or not. For example, if $\sum_{j=1}^N\eta_{j,0}=0$, for all $n\geq 1$, we have $\langle S_n,1 \rangle=  \langle S_n,1 \rangle_{\mu}=0$,  which yields that $S_n$ has at least one sign change on $(-1,1)$. On the other hand, if $\langle f,g \rangle=  \int_{-1}^{1} f(x) g(x) \,dx +f(6)g(6)$,  then $S_1(z)=z-2$, which is negative on  $(-1,1)$.

 As will be seen  in Lemma \ref{AsymCBounded_Lemma}, for sequentially-ordered discrete Sobolev inner products, the corresponding  orthogonal polynomial $S_n$ with degree  $n$ sufficiently large, has all its zeros  real and  simple, each  sufficiently small neighborhood of  $c_j$ ($j=1,\dots,N$) contains exactly one zero of $S_n$, and   from the Theorem \ref{Th_ZerosSimp} the remaining $n-N$ zeros lie on $(-1,1)$.

Let $\{Q_{n}\}_{n\geq 0}$ be the sequence of monic orthogonal polynomials with respect to the inner product
\begin{align}\label{Mod-InnerP}
  \langle f, g \rangle_{\rho} = & \int_{-1}^{1} f(x)\;g(x)\;d\mu_{\rho}(x), \; \text{ where } \; \rho(z)= \prod_{c_j<-1}\!\!\left( z-c_j\right)^{d_{j}+1} \!\prod_{c_j>1}\!\left(c_j-z\right)^{d_{j}+1} \\ \nonumber
   &\text{ and } \; d\mu_{\rho}(x)=\rho(x) d\mu(x).
\end{align}
Note that $\rho$ is a polynomial of degree $d=N+\sum_{j=1}^{N}\!d_j$ and  positive  on $[-1,1]$.

Now, we associate to the sequence $\{S_n\}_{n=0}^{\infty}$ the next  sequences of polynomials
\begin{equation}
\label{kth-AssPoly}
 S^{[k]}_{n}(z)  = \int_{-1}^{1} \frac{S_{{n+k}}(z)-S_{{n+k}}(x)}{z-x}\; Q_{k-1}(x)  \;d\mu_{\rho}(x),
\end{equation}  for $k\in \NN $  and $n \geq 0$.
Additionally, we adopt the convention   $S^{[0]}_{n}=S_{n}$. We call $\dsty \left\{S^{[k]}_{n}\right\}_{n=0}^{\infty}$ the   sequence of \emph{$k$th polynomials associated to} $\dsty \left\{S_{n}\right\}_{n=0}^{\infty}$.

As far as we know, the only extension of Markov's Theorem  for Sobolev orthogonal polynomials appears in  \cite[Th. 5.5]{MaMePi13}, when  the inner product \eqref{IP-Sobolev} is such that $N=1$, $d_1=1$, $c_1=0$, $\eta_{1,0}=0$,  and $\eta_{1,1}>0$. The main aim of the present paper is to prove the following theorem,   which provides a natural extension of the Markov's Theorem   for  the Sobolev case.

\begin{theorem}[Extended Markov's Theorem] \label{Ext_Markov_Th} Let \eqref{IP-Sobolev-SO}  be a sequentially-ordered discrete Sobolev inner product   with $\mu\in\NC$. Then,
for $k\in \NN$,
\begin{equation}\label{Th.Markov}
  R_{n}^{[k]}=\frac{S^{[k]}_{n}(z)}{S_{n+k}(z)} \unifn \widehat{\mu}_{k}(z)= \int_{-1}^{1} \frac{Q_{k-1}(x)}{z-x} d\mu_{\rho}(x), \quad K \subset  \Omega_{\infty}^*=\Omega_{\infty} \setminus \{c_1,c_2,\dots, c_N\}.
\end{equation}
We call  $\widehat{\mu}_{k}$ the \emph{$k$th Markov-type function  associated with $\mu_{\rho}$}.
\end{theorem}

Also, in Corollary \ref{CoroSpeed},  we give   the following estimate for the degree of convergence of the sequence of rational functions $\{R^{[k]}_{n}\}$ to the corresponding Markov-type function $\widehat{\mu}_{k}$.
$$
  \limsup_n \left\|\widehat{\mu}_{k} - R^{[k]}_n\right\|_K^{1/2n}\leq \|\varphi\|_K^{-1}<1,\quad \text{where $\dsty \|f\|_{K}= \sup_{z\in K}|f(z)|$}.
$$
The rest of the paper is organized as follows. The next section  is devoted to the consequences of the quasi-orthogonality of $S_n$ with respect to the measure $\mu$. Sections \ref{Sect-ProofTh1} and \ref{Sect-ProofTh2}  contain the proofs of Theorems \ref{Th_ZerosSimp}  and \ref{Ext_Markov_Th}  respectively, as well as some of their consequences. The Section \ref{Sec-AuxLemmas} deals with the auxiliary results for the proof of the main result (Theorem \ref{Ext_Markov_Th}).

\section{Recurrence relations}\label{Sec-RecRelat}

Unlike the rest of the paper, the inner product \eqref{IP-Sobolev} does not necessarily have to be sequentially-ordered in this section.

If $n >d$, from \eqref{Sobolev-Orth}, we have that $S_n$ satisfies the following quasi-orthogonality relations with respect to $d\mu_{\rho}$
\begin{equation}\label{quasi-orthogonal}
\langle S_n, f \rangle_{\rho}  =\langle S_n, \rho f \rangle_{\mu} =\int_{-1}^{1}S_n(x) \, f(x) \rho \, d\mu(x) = \langle S_n,\rho\;  f \rangle= 0 ,
\end{equation}
for all  $f\in \PP_{n-d-1}$, where $\PP_n$ is  the linear space of polynomials with real coefficients and degree at most $n\in \ZZp$.
Hence, \emph{the polynomial $S_n$ is quasi-orthogonal of order $d$ with respect to $d\mu_{\rho}$} and by this argument we get the next result.

\begin{proposition}\label{localzeros} Let $S_n$ be the $n$-th  orthogonal polynomial with respect to   \eqref{IP-Sobolev} and  $n>d$, then $S_n$ has at least $(n-d)$ changes of sign on $(-1,1)$.
\end{proposition}

\begin{proposition}\label{PropoGradoS_n}
 Let $S^{[k]}_{n}$ be the $k$th associated polynomial defined by \eqref{kth-AssPoly}.  Then  $S^{[k]}_{n}$ is a polynomial of degree $n$ and leading coefficient equal to $\|Q_{k-1}\|^2_{\mu_\rho}$.
\end{proposition}
\begin{proof}
Let $\dsty S_{n+k}(x)=\sum_{i=0}^{n+k}\theta_i \; x^i$ where $\theta_{n+k}=1$, then
\begin{align*}
 S^{[k]}_{n}(z) =&  \int_{-1}^{1} \frac{S_{{n+k}}(z)-S_{{n+k}}(x)}{z-x}\; Q_{k-1}(x)  \;d\mu_{\rho}(x)  = \sum_{i=1}^{n+k}\theta_i\int_{-1}^{1} \frac{z^i-x^i}{z-x}\; Q_{k-1}(x)  \;d\mu_{\rho}(x)\\
=& \sum_{i=1}^{n+k}\theta_i\int_{-1}^{1} \left(\sum_{j=0}^{i-1}z^{i-1-j}x^{j}\right)\; Q_{k-1}(x)  \;d\mu_{\rho}(x) \\
= &\sum_{i=1}^{n+k}\theta_i\sum_{j=0}^{i-1}z^{i-1-j}\left(\int_{-1}^{1} x^{j}\; Q_{k-1}(x)  \;d\mu_{\rho}(x)\right) = \sum_{i=1}^{n+k}\theta_i\sum_{j=k-1}^{i-1}\langle x^{j},Q_{k-1} \rangle_{\rho}  \; z^{i-1-j} \\
=& \sum_{j=k-1}^{n+k-1} \langle x^{j},Q_{k-1} \rangle_{\rho}  \;  z^{n+k-1-j}+ \sum_{i=1}^{n+k-1}\theta_i\sum_{j=k-1}^{i-1}\langle x^{j},Q_{k-1} \rangle_{\rho}  \; z^{i-1-j}
  \\
=&\langle x^{k-1},Q_{k-1} \rangle_{\rho}  \;  z^{n}+ f_{n-1}(z)  =\|Q_{k-1}\|_{\mu_\rho}^2  z^{n}+f_{n-1}(z).
\end{align*}
where  $f_{n-1}$ is a polynomial of degree at most $n-1$.
\end{proof}

In the standard case of orthogonality,  where the polynomials $\{P_n\}$ satisfy the three terms  recurrence relation \eqref{Standard-3TRR},  the sequence of associated polynomials $\{P_n^{[1]}\}$  can be generated by the recurrence relation \eqref{Standard-3TRR-2}. The following proposition is an analogous result for the sequence of associated polynomials $\{S_n^{[k]}\}$.

\begin{proposition}[Recurrence relation] For  $n \geq{2d-1} $, the   sequences $\{ S^{[k]}_n\}_{n=0}^{\infty}$    satisfy the following   $2d+1$ term recurrence relation
\begin{equation}\label{RecRel-AssoPoly}
   \rho(z) S^{[k]}_{n}(z)=  \sum_{j=n-d}^{n+d}\mathfrak{a}_{n+k,j+k}  \, S^{[k]}_{j}(z),\quad \text{where} \quad \mathfrak{a}_{n+k,j+k}=\frac{\langle S_{n+k},\rho S_{j+k}\rangle}{\langle S_{j+k},S_{j+k}\rangle}.\end{equation}
\end{proposition}

\begin{proof} It is straightforward to obtain \eqref{RecRel-AssoPoly} for  $k=0$ as a consequence of  \eqref{quasi-orthogonal}, i.e.,
 \begin{equation}\label{Sob-RecRelation}
 \rho(z) S_{n}(z)= \sum_{j=n-d}^{n+d}\mathfrak{a}_{n.j} \, S_j(z),\quad \text{where } \mathfrak{a}_{n.j}=\frac{\langle S_{n},\rho S_j\rangle}{\langle S_j,S_j\rangle}.
\end{equation}
Hence,  if $k>0$
\begin{align*}
  \frac{\rho(z) S_{n+k}(z) - \rho(x) S_{n+k}(x) }{z-x} \; Q_{k-1}(x)  = &  \sum_{j=n-d}^{n+d}\mathfrak{a}_{n+k.j+k} \,\frac{S_{j+k}(z)-  S_{j+k}(x)}{z-x}\; Q_{k-1}(x),\\
\int_{-1}^{1} \frac{\rho(z) S_{n+k}(z) - \rho(x) S_{n+k}(x) }{z-x} \; Q_{k-1}(x)   d\mu_{\rho}(x)= & \sum_{j=n-d}^{n+d}\mathfrak{a}_{n+k.j+k}  \, S^{[k]}_{j}(z).
\end{align*}
As  $n \geq{2d-1}$, from \eqref{quasi-orthogonal}, we get $\dsty \int_{-1}^{1} S_{n+k}(x) \;\frac{\rho(z)- \rho(x) }{z-x} \; Q_{k-1}(x)   d\mu_{\rho}(x)=0.$ Hence,
\begin{align*}
\rho(z)  S^{[k]}_{n}(z)= & \int_{-1}^{1} \frac{\rho(z)( S_{n+k}(z) - S_{n+k}(x))}{z-x} \; Q_{k-1}(x)   d\mu_{\rho}(x)
\\ & + \int_{-1}^{1} S_{n+k}(x) \;\frac{\rho(z)- \rho(x) }{z-x} \; Q_{k-1}(x)   d\mu_{\rho}(x)\\
   = &\int_{-1}^{1} \frac{\rho(z) S_{n+k}(z) - \rho(x) S_{n+k}(x) }{z-x} \; Q_{k-1}(x)   d\mu_{\rho}(x),
\end{align*}
and we get \eqref{RecRel-AssoPoly}.
\end{proof}

Remember   that    $\{Q_{n}\}_{n\geq 0}$ is the sequence of monic orthogonal polynomials with respect to $d\mu_{\rho}$, which was defined in  \eqref{Mod-InnerP}. As it  is known,  this sequence  satisfies the three-term  recurrence relation
\begin{equation}\label{RR3T-Mod}
Q_{n+1}(z)= (z-\beta_{n}) Q_{n}(z)-\alpha_n^2 \;Q_{n-1}(z),\quad n \geq 0,
\end{equation}
where $Q_{-1}=0$, $Q_{0}=1$, $\|\cdot \|_{\mu_{\rho}}^2=\langle \cdot ,\cdot \rangle_{\rho}$, $\beta_{n}={\langle Q_{n} ,x  Q_{n} \rangle_{\rho}}/{\|Q_{n}\|_{\mu_{\rho}}^2} $,  $\alpha_n= {\|Q_{n}\|_{\mu_{\rho}}}/{\|Q_{n-1}\|_{\mu_{\rho}} }$ and $\dsty \alpha^2_0=\int_{-1}^1d\mu_{\rho}(x)$.

Following  \cite{Wal91}, we define its $k$th  sequence of associated polynomials  $\{Q_n^{[k]}\}$  ($k\in \ZZp$) as
\begin{equation}\label{Q-Asociated}
 Q_n^{[k]}(z)=\int_{-1}^{1} \frac{Q_{n+k}(z)-Q_{n+k}(x)}{z-x} \; Q_{k-1}(x) d\mu_{\rho}(x),
\end{equation}
where $ Q_n^{[0]}= Q_n$. Note that $Q_n^{[k]}$ is a polynomial in $z$ of degree $n$. From \cite[(1.3) and (2.13)]{Wal91}
\begin{equation}\label{Q[k]_3TRR}
{Q}_{n+1}^{[k]}(x)= (x-\beta_{n+k}) Q^{[k]}_{n}(x) - \alpha_{n+k}^2 Q^{[k]}_{n-1}(x).
\end{equation}

The next proposition is analogous to \cite[(2.5)]{Wal91} for the Sobolev case.

\begin{proposition} For   ${n\geq d-1}$, the   sequences  $\{ S^{[k]}_n\}_{n=0}^{\infty}$, for
 $\;k \geq 2$,  hold the following  relation
   \begin{equation} \label{StrRel-AssoPoly}
    S^{[k]}_{n}(z)  =  (z-\beta_{k-2}) S^{[k-1]}_{n+1}(z) - {\alpha^2_{k-2}}    S^{[k-2]}_{n+2}(z).
\end{equation}
\end{proposition}

\begin{proof}  From \eqref{RR3T-Mod}-\eqref{Q-Asociated},
\begin{align}\nonumber
   S^{[k]}_{n}(z)  = &  \int_{-1}^{1} \frac{S_{n+k}(z)-S_{n+k}(x)}{z-x}\; \left((x-\beta_{k-2}) Q_{k-2}(x)- \alpha^2_{k-2} Q_{k-3}(x)\right)  \;d\mu_{\rho}(x) \\
   =& \begin{cases}
 \dsty  \int_{-1}^{1} \frac{S_{n+k}(z)-S_{n+k}(x)}{z-x}  x Q_{k-2}(x)  \;d\mu_{\rho}(x)   - \beta_{k-2} S^{[k-1]}_{n+1}(z) - \alpha^2_{k-2}   S^{[k-2]}_{n+2}(z), &  \text{if } k\geq 3,\\
 \dsty   \int_{-1}^{1} \frac{S_{n+2}(z)-S_{n+2}(x)}{z-x}  x d\mu_{\rho}(x)-\beta_0S^{[1]}_{n+1}(z), &  \text{if } k=2.\\
  \end{cases} \label{StrRel-AssoPoly-1}
  \end{align}
  From orthogonality,
$$
\int_{-1}^{1}  \frac{S_{n+k}(z)-S_{n+k}(x)}{z-x} \;(z-x) \;  Q_{k-2}(x)  \;d\mu_{\rho}(x)
=\begin{cases}0,&\text{ if } k\geq 3,\\
\alpha^2_0  S_{n+2}(z)
,&\text{ if } k=2.
\end{cases}$$
Therefore,
\begin{equation}\label{StrRel-AssoPoly-2}
 \int_{-1}^{1} \frac{S_{n+k}(z)-S_{n+k}(x)}{z-x}x Q_{k-2}(x)d\mu_{\rho}(x) =\begin{cases}  z  S^{[k-1]}_{n+1}(z),& \ \text{if } k\geq 3,\\
z  S^{[1]}_{n+1}(z)-\alpha^2_0S_{n+2}(z), &    \ \text{if } k=2.
 \end{cases}
\end{equation}
Substituting \eqref{StrRel-AssoPoly-2}  into \eqref{StrRel-AssoPoly-1}, we get \eqref{StrRel-AssoPoly}.
\end{proof}

\section{Proof of Theorem \ref{Th_ZerosSimp}}\label{Sect-ProofTh1}

In the remainder of the paper, we assume that \eqref{IP-Sobolev} is sequentially-ordered. Therefore, we can rewrite \eqref{IP-Sobolev}   as \eqref{IP-Sobolev-SO} with $0\leq d_1\leq d_2\leq \cdots\leq d_N$.
The next lemma is an extension of \cite[Lemma 2.1]{LoPiPe01}.

\begin{lemma}\label{Lem-CoroRolle}  Let $L$ be a polynomial with real coefficients  of degree  $\geq m\in \NN$, $\{\Delta_i\}_{i=0}^m$ be a set of intervals on the real line, and $I_k=\ch{\cup_{i=0}^{k} \Delta_i}$ for  $k=0,1,\dots, m$.
If
\begin{equation}\label{Lem-CoroRolle-1}
I_{k-1}\cap \Delta_{k}=\emptyset, \quad \quad k=1,2,\dots, m;
\end{equation}
 then
\begin{equation}\label{Lem-CoroRolle-2}
\sum_{i=0}^m \Nzero\left(L^{(i)};\Delta_i\right) \leq \Nzero\left(L^{(m)};I_m\right)+m \leq \dgr{L},
\end{equation}
where for a given  non-null polynomial $f$ and $A \subset \RR$ the symbol $\Nzero(f;A)$ denotes the total number of zeros (counting multiplicities) of $f$ on $A$.
\end{lemma}
\begin{proof}  For   $m=0$,  it is straightforward that  $\Nzero(L;\Delta_0) \leq \Nzero(L;\Delta_0)+0\leq \dgr{L}$. We now proceed by induction on $m$.  Suppose that we have  $\kappa+1$ intervals  $\{\Delta_i\}_{i=0}^\kappa$ that satisfy \eqref{Lem-CoroRolle-1}, and that \eqref{Lem-CoroRolle-2} is  true for the first $\kappa-1$ intervals.

From   Rolle's Theorem,  $ \Nzero(f;A)\leq  \Nzero(f^{\prime};A)+1$, where $A$ is an interval of the real line and $f^{\prime}$ a non-null polynomial with real coefficients. Therefore,
\begin{align*}
 \sum_{i=0}^{\kappa} \Nzero\left(L^{(i)};\Delta_i\right)=&\sum_{i=0}^{\kappa-1} \Nzero(L^{(i)};\Delta_i)+\Nzero\left(L^{(\kappa)};\Delta_{\kappa}\right)\\
 \leq & \Nzero\left(L^{(\kappa-1)};I_{\kappa-1}\right)+(\kappa-1)+\Nzero\left(L^{(\kappa)};\Delta_{\kappa}\right)\\
\leq & \Nzero\left(L^{(\kappa)};I_{\kappa-1}\right)+1+\Nzero\left(L^{(\kappa)};\Delta_{\kappa}\right)+(\kappa-1)\\
\leq & \Nzero\left(L^{(\kappa)};I_{\kappa-1}\cup \Delta_{\kappa}\right)+\kappa \leq \Nzero\left(L^{(\kappa)};I_{\kappa}\right)+\kappa \leq \dgr{L}.
\end{align*}\end{proof}

\begin{lemma}\label{PolMiDeg}
Let $\{(r_i,\nu_i)\}_{i=1}^M$ be a sequence of $M$ ordered pairs which is sequentially-ordered. Then, there exists   a   unique monic polynomial   $U_M$ of minimal degree,  such that
\begin{align}\label{cond}
U_M^{(\nu_i)}(r_i)=0 \quad \text{ for }i=1,2,\dots,M.
\end{align}
Furthermore, the degree of $U_M$ is  $\kappa_M=\min \,\mathfrak{I}_{M}-1$, where $\mathfrak{I}_{M}=\{i: 1\leq i\leq M\; \text{ and }\; \nu_i \geq i\}\cup \{M+1\}$.
\end{lemma}

\begin{proof}
The existence of a  not identically zero polynomial with degree $\leq M$ satisfying \eqref{cond}   reduces to solving a homogeneous linear system of $M$ equations on $M+1$ unknowns (its coefficients). Thus, a non trivial solution always exists. In addition, if we suppose that there exist two different minimal monic polynomials $U_M$ and $\widetilde{U}_M$, then  the polynomial $\widehat{U}_M=U_M-\widetilde{U}_M$ is not identically zero, it satisfies \eqref{cond}, and  $\dgr{\widehat{U}_M}<\dgr{U_M}$. So, if we divide $\widehat{U}_M$ by its leading coefficient, we reach a contradiction.

The rest of the proof runs  by induction on the number of points $M$. For $M=1$, the result follows taking
$$U_1(x)=\begin{cases}x-r_1&,  \text{ if } \nu_1=0,\\
1&, \text{ if } \nu_1\geq 1
.\end{cases}$$

 Suppose that,  for each sequentially-ordered sequence  of $M$ ordered pairs,  the corresponding minimal  polynomial $U_M$ has degree  $\kappa_M$.

 Let $\{(r_i,\nu_i)\}_{i=1}^{M+1}$ be  a sequentially-ordered sequence of $M+1 $ ordered pairs. Obviously,  $\{(r_i,\nu_i)\}_{i=1}^{M}$ is also sequentially-ordered, $\dgr{U_{M+1}}\geq \dgr{U_M}$,  and from the induction hypothesis $\dgr{U_M}=\kappa_M$.  Now, we shall divide the proof in two cases:

\begin{enumerate}
\item If $\kappa_{M+1}=M+1$, then for all $1\leq i\leq M+1$ we have $ \nu_i< i$, which yields
\begin{equation}\label{PolMiDeg-1}
\dgr{U_{M+1}}\geq \dgr{U_M}=\kappa_M=M \geq \nu_{M+1}.
\end{equation}
Let $\dsty \Delta_{k}=\ch{\{c_i:\nu_i=k\}}$ for $k=0,1,2,\dots, \nu_{M+1}$. As $\{(r_i,\nu_i)\}_{i=1}^{M+1}$ is sequentially-ordered, the set of intervals $\{\Delta_{k}\}_{k=0}^{\nu_{M+1}}$ satisfy \eqref{Lem-CoroRolle-1}. Therefore, from \eqref{PolMiDeg-1} and Lemma \ref{Lem-CoroRolle} we get
$$
M+1\leq \sum_{i=0}^{\nu_{M+1}} \Nzero\left(U_{M+1}^{(i)};\Delta_i\right) \leq  \dgr{U_{M+1}},
$$  which implies that  $\dgr{U_{M+1}}=M+1=\kappa_{M+1}$.

\item  If $\kappa_{M+1}\leq M$, then  there exists a minimal $j$ ($1\leq j\leq M+1$),  such that $\nu_j\geq j$,  and $ \nu_i< i$ for all $1\leq i\leq j-1$. Therefore,  $\kappa_{M+1}=j-1=\kappa_M$. From the induction hypothesis
$$\dgr{U_M}=\kappa_M=j-1\leq \nu_j-1\leq \nu_{M+1}-1,$$ which gives  $U^{(\nu_{M+1})}_M\equiv 0$. Hence,  $U_{M+1}\equiv U_{M}$ and $\dgr{U_{M+1}}=\dgr{U_{M}}=\kappa_M=\kappa_{M+1}$.
\end{enumerate}\end{proof}

Observe  that, in Lemma \ref{PolMiDeg},  the assumption of $\{(r_i,\nu_i)\}_{i=1}^M$ being sequentially-ordered is necessary for asserting that the polynomial $U_M$ has degree $\kappa_M$. In fact, if we consider the non sequentially-ordered sequence $\{(-1,0),(1,0),(0,1)\}$, we get $U_3=x^2-1$ and $\kappa_3=3\neq \dgr{U_3}$.

\begin{proof}[Proof of Theorem \ref{Th_ZerosSimp}]\
 From the sequentially-ordered conditions,  the intervals $$\Delta_{0}=\ch{(-1,1)\cup\{c_i:d_i=0\}} \quad, \quad   \Delta_{k}=\ch{\{c_i:d_i=k\}} \quad\text{for }\; k=1,2,\dots, N,$$
satisfy \eqref{Lem-CoroRolle-1}.

Let  $\xi_1<\xi_2<\cdots <\xi_{\ell}$ be the points on $(-1,1)$ where $S_n$ changes sign and suppose that $\ell<n-N$. Let $\{(r_i,\nu_i)\}_{i=1}^{N+\ell}$ be the  sequentially-ordered sequence $$(r_i,\nu_i)=\left\{
                \begin{array}{ll}
                  (\xi_i,0), & \hbox{if } i=1,2,\dots,\ell;\\
                  (c_{i-\ell},d_{i-\ell}), & \hbox{if }  i=\ell+1, \ell+2,\dots,\ell+N.
                \end{array}
              \right.
$$
From Lemma \ref{PolMiDeg}, there exists a unique monic polynomial   $U_{N+\ell}$ of minimal degree,  such that
\begin{align*}
U_{N+\ell}^{(\nu_i)}(r_i)=0;&\qquad\text{for }\; i=1,\dots, N+\ell.
\end{align*}
Furthermore,
\begin{equation}\label{DegQ}
\dgr{U_{N+\ell}}=\min \,\mathfrak{I}_{N+\ell}-1\leq N+\ell,
\end{equation}
where  $\mathfrak{I}_{N+\ell}=\{i: 1\leq i\leq N+\ell\; \text{ and }\; \nu_i \geq i\}\cup \{N+\ell+1\}$.  Now, we need to consider the following two cases.

\begin{enumerate}
\item If $\dgr{U_{N+\ell}}<N+\ell$,  from \eqref{DegQ},  there exists $1\leq j\leq N+\ell$ such that $\dgr{U_{N+\ell}}=j-1$, $\nu_{j}\geq j$ and $\nu_i\leq i-1$ for $i=1,2,\dots,j-1$. Hence, $\nu_{j-1}+1\leq j-1=\dgr{U_{N+\ell}}$. Thus, from Lemma \ref{Lem-CoroRolle},
$$
j-1\leq\sum_{k=0}^{\nu_{j-1}} \Nzero\left(U_{N+\ell}^{(k)};\Delta_k\right) \leq  \dgr{U_{N+\ell}}=j-1,
$$
\item If $\dgr{U_{N+\ell}}=N+\ell$, from \eqref{DegQ},  we get $\dgr{U_{N+\ell}}=N+\ell\geq\nu_{\ell+N}+1=d_N+1$ and from Lemma \ref{Lem-CoroRolle},
$$
N+\ell\leq\sum_{k=0}^{d_N} \Nzero\left(U_{N+\ell}^{(k)};\Delta_k\right) \leq  \dgr{U_{N+\ell}}=N+\ell,
$$

\end{enumerate}
In both cases, we obtain that $U_{N+\ell}$ has  simple zeros on $(-1,1) \subset \Delta_0$ and has no other zeros than those given by construction. Now, since $\dgr{U_{N+\ell}}\leq \ell+N<n$, we arrive at the contradiction
\begin{align*}
 0&=\langle S_n,U_{N+\ell} \rangle=\int_{-1}^{1} S_n(x) U_{N+\ell}(x) \, d\mu(x) +\sum_{j=1}^{N}\eta_{j,d_j} S_n^{(d_j)}(c_{j}) U_{N+\ell}^{(d_j)}(c_{j})\\
 & =\int_{-1}^{1} S_n(x) U_{N+\ell}(x) \,  d\mu(x)\neq0.
\end{align*}
\end{proof}

The following Lemma is  a direct consequence of  \cite[(1.10)]{LopMarVan95}, when  instead of the inner product  \cite[(1.1)]{LopMarVan95}, we consider \eqref{IP-Sobolev-SO}.

\begin{lemma} \label{Lago_Lemma} Consider the sequentially-ordered inner product \eqref{IP-Sobolev-SO}   with $\mu\in\NC$.   Then,
  \begin{equation}\label{AympComp}
\frac{S_n(z)}{P_n(z)} \unifn \prod_{j=1}^{N} \frac{(\varphi(z)-\varphi(c_j))^2}{2 \varphi(z)\;(z-c_j)}, \quad  K\subset  \overline{\CC}\setminus [-1,1],
  \end{equation}
  where  $\varphi$ is as in \eqref{AympRatioStandard}.
   \end{lemma}

  Now, combining Theorem \ref{Th_ZerosSimp}  and Lemma \ref{Lago_Lemma}, we get the following  useful lemma.

\begin{lemma} \label{AsymCBounded_Lemma} If \eqref{IP-Sobolev-SO} is a sequentially-ordered Sobolev inner product  such that $\mu\in\NC$,   then:
\begin{enumerate}
\item For all $n$ sufficiently large,   each sufficiently small neighborhood of  $c_j$; $j=1,\dots,N$; contains exactly one zero of $S_n$, and the remaining $n-N$ zeros lie on $(-1,1)$.
\item For all $n$ sufficiently large,  the zeros of $S_n$ are real and simple.
\item The set of zeros of $\{S_n\}_{n=1}^\infty$ is  uniformly bounded.
\end{enumerate}
   \end{lemma}

\begin{proof} The first assertion of the lemma is a direct consequence of \eqref{AympComp}  and   Rouch\'{e}'s Theorem (see \cite[Th. 9.2.3]{Hille82}). Note that $S_n$ is a polynomial with real coefficient. Therefore, the second and third sentences are consequences of the first assertion and Theorem  \ref{Th_ZerosSimp}.
\end{proof}

\section{Auxiliary lemmas}\label{Sec-AuxLemmas}

Let $S_n$ be the $n$-th  orthogonal polynomial with respect to  the sequentially-ordered inner product \eqref{IP-Sobolev-SO}. Taking into consideration the Theorem \ref{Th_ZerosSimp}, let $\{\xi_{n,i}\}_{i=1}^{n-N}$ be the $n-N$ simple zeros of $S_n$  on $(-1,1)$ for all sufficiently large $n$ and let $\{\xi_{n,n-N+i}\}_{i=1}^{N}$ be  the remaining $N$ zeros of $S_n$. Obviously, $S_n$ admits the representation
\begin{equation}\label{S-Decomposition}
S_{n}(x)= S_{n,1}(x)\;  S_{n,2}(x), \; \text{ where }\;  S_{n,1}(x)=\prod_{i=1}^{n-N}(x-\xi_{n,i}) \text{ and  }  S_{n,2}(x)=\prod_{i=1}^{N}(x-\xi_{n,n-N+i}).
\end{equation}  From Lemma \ref{AsymCBounded_Lemma}, for all sufficiently large $n$, the last $N$ zeros of $S_n$ are real and simple. Furthermore, the sign of $S_{n,2}$ is constant on $[-1,1]$ and equal to $(-1)^{\nu}$, where $\nu$ is  the number of $c_j$ greater  than $1$. Thus, the polynomial $ \dsty S^+_{n,2}(x)=(-1)^{\nu} S_{n,2}(x)$ is positive on $[-1,1]$.

The following Lemma  is an analogous of the Gauss-Jacobi quadrature formula for the sequentially-ordered Sobolev inner product, when $n$ is sufficiently large.

\begin{lemma}\label{Th-GaussJacobiType-QF}
Let $S_n$ and $\{\xi_{n,i}\}_{i=1}^{n-N}$ as above. If  $n$ is sufficiently large, then for every polynomial $T$ with $\dgr{T}\leq 2n-d-N-1$,
\begin{align}\label{GaussJacobiType-QF}
  \int_{-1}^{1}  T(x)S^{{+}}_{n,2}(x) d\mu_{\rho}(x)=&  \sum_{i=1}^{n-N} \lambda_{n,i}\,S^{+}_{n,2}(\xi_{n,i})\,T(\xi_{n,i}),  \\ \nonumber
  &  \text{ where } \; \lambda_{n,i}=\int_{-1}^{1} \frac{S_n(x)}{S'_{n}(\xi_{n,i})(x-\xi_{n,i})}d\mu_{\rho}(x).
\end{align}
Moreover,  the number of positive coefficients $\lambda_{n,i}$ is greater than or equal to $\left(n-\frac{d+N}{2}\right)$. We  call \emph{Christoffel-type coefficients}  to the numbers $\dsty \{\lambda_{n,i}\}_{i=1}^n$.
\end{lemma}

\begin{proof}
Let $T$ be an arbitrary polynomial of degree at most $2n-d-N-1$ and denote by $\mathcal{L}$ the Lagrange polynomial interpolating $T$ at the points $\xi_{n,1},\dots,\xi_{n,n-N}$ ($\dgr{\mathcal{L}}<n-N$), i.e.,
$$\mathcal{L}(z)=\sum_{i=1}^{n-N}T(\xi_{n,i}) \, \frac{S_{n,1}(z)}{S_{n,1}^{\prime}(\xi_{n,i})(z-\xi_{n,i})}\,.$$
 Then, $T-\mathcal{L}=f\; S_{n,1}$ where $\dgr{f}\leq n-d-1$. From \eqref{quasi-orthogonal}
$$
\int_{-1}^{1} (T-\mathcal{L})(x)\, S_{n,2}(x)\, d\mu_{\rho}(x)= \int_{-1}^{1} f(x) \,S_n(x)\,   d\mu_{\rho}(x)= 0.
$$
Hence,
\begin{align*}
\int_{-1}^{1} T(x)\,S_{n,2}(x)\, d\mu_{\rho}(x)=&\int_{-1}^{1} \mathcal{L}(x)\,S_{n,2}(x)\,  d\mu_{\rho}(x),\\
= & \int_{-1}^{1} \left( \sum_{i=1}^{n-N}T(\xi_{n,i}) \, \frac{S_{n,1}(x)}{S_{n,1}^{\prime}(\xi_{n,i})(x-\xi_{n,i})}\right) S_{n,2}(x)d\mu_{\rho}(x),\\
=&\sum_{i=1}^{n-N} \left(\int_{-1}^{1} \frac{S_n(x)}{S'_{n,1}(\xi_{n,i})(x-\xi_{n,i})}d\mu_{\rho}(x)\right)\;T(\xi_{n,i}),
\end{align*}
which establishes  \eqref{GaussJacobiType-QF}. Assume that $n$ is fixed,  let $I_+ =\{1\leq i\leq n-N: \lambda_{n,i}>0\}$ and  $\dsty \Lambda_{+}^2(x)= \prod_{i\in I_+}(x-\xi_{n,i})^2$. If $\dgr{ \Lambda_{+}^2} <2n - d-N$, from  \eqref{GaussJacobiType-QF},
$$ 0 < \int_{-1}^{1} \Lambda_{+}^2(x)\; S^+_{n,2}(x)d\mu_{\rho}(x)  = \sum_{\overset{i=1}{i\not \in I_+}}^{n-N} \lambda_{n,i}\;\Lambda_{+}^2(\xi_{n,i})\;S^+_{n,2}(\xi_{n,i})\leq 0, $$
 which is a contradiction and the second assertion  is established.\end{proof}

 Let us denote for $ k\in \NN$
\begin{align}\label{RationalFunct-1}
  R^{[k]}_{n,1}(z)=\frac{S^{[k]}_{n,1}(z)}{S_{n+k,1}(z)},& \text{ where }  S^{[k]}_{n,1}(z)= \int_{-1}^{1} \frac{S_{n+k,1}(z)-S_{n+k,1}(x)}{z-x}\;Q_{k-1}(x)  \;d\mu_{\rho,n}(x) \\ \nonumber
&  \text{ and }\;d\mu_{\rho,n}(x)= S^+_{n+k,2}(x)\;\rho(x) \;d\mu(x).
\end{align}
From Lemma \ref{AsymCBounded_Lemma}, it is straightforward to see that:
\begin{enumerate}
  \item  If $n$ is sufficiently large, $S^+_{n+k,2}(x)\;\rho(x)>0$ for all $x\in [-1,1]$.
  \item  There exists a constant $\mathfrak{M}_{\rho}>0$, such that for all $n\in \ZZp$
  \begin{equation}\label{MesureBounded}
  \mu_{\rho,n}([-1,1])=  \int_{-1}^{1}S^+_{n+k,2}(x)\;\rho(x) \;d\mu(x) \leq \mathfrak{M}_{\rho}.
  \end{equation}
\end{enumerate}

\begin{lemma}[Principal Lemma]\label{Lem-Acota-lambda} Let $\{S_n\}_{n=0}^{\infty}$ be the monic orthogonal polynomial sequence with respect to a sequentially-ordered Sobolev inner product ~\eqref{IP-Sobolev-SO}. Then, for $n$ sufficiently large
\begin{equation}\label{Th.DescFracSim-1}
  R^{[k]}_{n,1}(z)=\sum_{j=1}^{n+k-N}\frac{S^+_{n+k,2}(\xi_{n+k,j})\;\lambda_{n+k,j}}{(z-\xi_{n+k,j})}.
\end{equation}

Furthermore, $\{ R^{[k]}_{n,1}\}$ is uniformly bounded   on each compact subset  $K \subset \CC\setminus[-1,1]$.
\end{lemma}

\begin{proof} Let $n$ and $k$ be fixed. For simplicity of notation, we  write $\xi_{j}$ instead of $\xi_{n+k,j}$. Then,  $\{\xi_{j}\}_{j=1}^{n+k-N}$ is the set of  zeros of $S_{n+k}$ on $(-1,1)$.

From Theorem \ref{Th_ZerosSimp}, for $n$ sufficiently large, we have that the zeros of $S_{n+k}$ are simple and  $n+k-N$ of them  lie on $(-1,1)$. Thus, $ S_{n+k}^\prime (\xi_{j}) \neq 0$ for $j=1,\dots,n+k-N$; and
\begin{equation*}
 R^{[k]}_{n,1}(z)=   \sum_{j=1}^{n+k-N} \frac{b_{j}}{z-\xi_{j}},
\end{equation*}
  where
\begin{align*}\nonumber
  b_{j}= & \lim_{z\to \xi_{j}}(z-\xi_{j})R^{[k]}_{n,1}(z)=\lim_{z\to\xi_{j}}\frac{(z-\xi_{j})}{S_{n+k,1}(z)}\lim_{z\to\xi_{j}} {S}^{[k]}_{n,1}(z) \\= & S_{n+k,2}(\xi_{j}) \int_{-1}^{1} \frac{(-1)^{\nu}S_{n+k}(x)\,  Q_{k-1}(x)\,   d\mu_{\rho}(x)}{ S_{n+k}^\prime (\xi_{j}) (x-\xi_{j})}  =  S^+_{n+k,2}(\xi_{j})\;\lambda_{n+k,j},
\end{align*}
and we get \eqref{Th.DescFracSim-1}.

The second part of this proof, as  \cite[Lemma 1]{Rakh77-2},  is based on the second  proof of  Chebyshev-Markov-Stieltjes's  Separation Theorem  in  \cite[\S 3.41]{Szg75}.
Through the proof, we  use the following notations:
\begin{align*}
    d\vartheta(x)  =  \dsty \sum_{j=1}^{n+k-N} \lambda_{n+k,j}\; S^+_{n+k,2}(\xi_{j}) \delta_{\xi_{j}}(x), & \; \;  \delta_{\xi_{j}}(x) = \begin{cases}1, \quad &x=\xi_{j},\\
0, \quad &x \neq \xi_{j}.
\end{cases}, \\
  \vartheta(x) = \dsty  \int_{-1}^{x}d\vartheta(t),   \; u_{\rho,n}(x) = \dsty  \int_{-1}^{x}d\mu_{\rho,n}(t) \; \text{ and }  &\; \;  \omega(x)  =  \dsty  u_{\rho,n}(x)-\vartheta(x).
  \end{align*}
Let us recall that the function  $u_{\rho,n}$ is  monotone nondecreasing on $[-1,1]$. Set $\xi_{0}=-1$ and $\xi_{n+k-N+1}=1$. Then, $ \vartheta$ is a step-function, which is constant on each interval $(\xi_{j},\xi_{j+1})$ for $j=0,1,\dots, n+k-N$. Hence, $\omega$ is  monotone nondecreasing  on each of these open intervals.

With these notations, we can rewrite \eqref{GaussJacobiType-QF} as
\begin{equation}\label{Lem-Acota-lambda-1}
  \int_{-1}^{1} T(x) \;d\omega(x)=0,\; \text{for any polynomial $T$ of degree at   most $(2(n+k)-d-N-1)$.}
\end{equation}

 As $\omega(-1)=u_{\rho,n}(-1)-\vartheta(-1)=0$ and  $$\omega(1)=u_{\rho,n}(1)-\vartheta(1)=\mu_{\rho,n}([-1,1])-\mu_{\rho,n}([-1,1])=0,$$  integrating by parts in \eqref{Lem-Acota-lambda-1}, we get
\begin{equation}\label{Lemm-GaussJacobiType-QF-3}
  \int_{-1}^{1} \omega(x)\; T^{\prime}(x) \;dx=0.
\end{equation}
We  use the symbol $\Nsign(q;I)$  to denote  the number of points of sign change of the function  $q$  on the interval $I\subset \RR$.  Obviously, in \eqref{Lemm-GaussJacobiType-QF-3}, the polynomial $T^{\prime}$ can be replaced by any other polynomial of degree at    most $(2(n+k)-d-N-2)$ and consequently, we can assert that $\Nsign(\omega;(-1,1)) \geq 2(n+k)-d-N-1$.

Note that  $\Nsign(\omega;(\xi_{0},\xi_{1}))=\Nsign(\omega;(\xi_{n+k-N},\xi_{n+k-N+1}))=0$.  Take into account that  $\omega$ is  monotone nondecreasing on each interval $(\xi_{j},\xi_{j+1})$, $j=1,\dots, n+k-N-1$. Hence, it has at most one sign change on each of them. Therefore, we can conclude that the total number of sign changes of $\omega$ on $\bigcup_{j=1}^{n+k-N-1}(\xi_{j},\xi_{j+1})$  is not greater than $(n+k-N-1)$. On the other hand, $\omega$  could change sign at each of the  $n + k-N$ points $\xi_{j}$.  In conclusion,
$$
 2(n+k-N)-(d-N)-1 \leq \Nsign(\omega;(-1,1))\leq 2(n + k-N)-1.
$$
It thus follows that  the number of intervals $(\xi_{j},\xi_{j+1})$ where $\omega$ does not change sign is at most $(d-N)$. Indeed, if the  number of intervals $(\xi_{j},\xi_{j+1})$ where $\omega$ does not change sign is at least $(d-N+1)$, then $  2(n+k)-d-N-1 \leq \Nsign(\omega;(-1,1))\leq 2(n + k)-1-d-N-2$, which is a contradiction.

We   say that $\xi_j \in E_1$ if the function $\omega$ changes sign in each of the consecutive  intervals $(\xi_{j-1},\xi_j)$ and $(\xi_{j},\xi_{j+1})$. In any other case, we  say that $\xi_j \in E_2$.

Observe that  if  $\omega$  does  not change sign on $(\xi_{j},\xi_{j+1})$,  then $\xi_{j},\xi_{j+1} \in E_2$. From the previous considerations, the number of  interval, where  $\omega$ does not change sign is at most $(d-N)$. Therefore,   $E_2$ cannot contain more than $2(d-N)$ elements.

Suppose that $\lambda_j \leq 0$. If $\xi_j \in E_1$, we know that  $\omega$ changes sign in each of the consecutive  intervals $(\xi_{j-1},\xi_j)$ and $(\xi_{j},\xi_{j+1})$. Let $x_1\in (\xi_{j-1},\xi_j)$ such that $\omega(x_1)>0$ and let $x_2\in (\xi_{j},\xi_{j+1})$ such that $\omega(x_2)<0$. As $u_{\rho,n}(x)$ is monotone nondecreasing on $(-1,1)$,  we get
$$ 0< \omega(x_1)-\omega(x_2)= \left(u_{\rho,n}(x_1)-u_{\rho,n}(x_2)\right)+ \lambda_j  \, S^+_{n+k,2}(\xi_{j}) \leq 0.$$
This contradiction proves that $\xi_j \in E_1$ implies that  $\lambda_j > 0$ (i.e., the Christoffel coefficients corresponding to the zeros  $\xi_j \in E_1$  are positive).

 Now, let  $\xi_j \in E_1$, $x_1\in (\xi_{j-1},\xi_j)$ such that $\omega(x_1)\leq 0$ and $x_2\in (\xi_{j},\xi_{j+1})$ such that $\omega(x_2)\geq 0$.  Recalling again that  $u_{\rho,n}(x)$ is monotone nondecreasing on $(-1,1)$,  then  $0 \geq  \omega(x_1)-\omega(x_2)= \left(u_{\rho,n}(x_1)-u_{\rho,n}(x_2)\right)+ \lambda_j  \, S^+_{n+k,2}(\xi_{j})$ and $
\lambda_j  \, S^+_{n+k,2}(\xi_{j}) \leq    u_{\rho,n}(x_2)-u_{\rho,n}(x_1) \leq  \mu_{\rho,n}(\xi_{j+1})-\mu_{\rho,n}(\xi_{j-1})$.  From the last inequality, we get
\begin{equation}\label{Lem-Acota-lambda-2}
  \sum_{\xi_j \in E_1}|\lambda_j \, S^+_{n+k,2}(\xi_{j})|=\sum_{\xi_j \in E_1}\lambda_j  \, S^+_{n+k,2}(\xi_{j}) \leq 2 \mu_{\rho,n}([-1,1]).
\end{equation}
  Set  $K \subset \CC\setminus[-1,1]$ compact  and  $\dsty \mathfrak{m}=\min_{\overset{x \in  [-1,1]}{z\in K}} |x-z|>0$, then
\begin{equation}\label{Lem-Acota-lambda-3}
\sum_{\xi_j \in E_1}\left| \frac{S^+_{n+k,2}(\xi_{j})\;\lambda_{n+k,j}}{(z-\xi_{j})} \right|  \leq   \frac{2 \mu_{\rho,n}([-1,1])}{ \mathfrak{m}}\leq \frac{2\mathfrak{M}_{\rho}}{\mathfrak{m}},
\end{equation}
where $\mathfrak{M}_{\rho}$ was defined in ~\eqref{MesureBounded}.

The aim of the last step of the proof is to show that the sum $\dsty G_{2}(z)=\sum_{\xi_j \in E_2}  \frac{S^+_{n+k,2}(\xi_{j})\;\lambda_{n+k,j}}{(z-\xi_{j})}  $ is  uniformly bounded on  $K$. We renumber the zeros of $S_{n+k,1}$ in such a way that   $E_2=\{\xi_1,\dots, \xi_m\} $ and $E_1=\{\xi_{m+1},\dots, \xi_{n+k-N}\}$. From the previous result,  $m \leq 2(d-N)$.

Firstly,   we introduce several  notations. Let   $\sigma_{\eta}$  be the $\eta$th  elementary symmetric polynomials evaluated  in $(\xi_1,\dots, \xi_m)$  (see \cite[(1.2.4)]{RahSch02}), i.e.,
\begin{align*}
  \sigma_{0}= \sigma_{0}(\xi_1,\dots, \xi_m)=& 1, \\
 \sigma_{\eta}= \sigma_{\eta }(\xi_1,\dots, \xi_m)= & \sum_{1\leq v_1<\cdots<v_{\eta}\leq m} \prod_{l=1}^{\eta} \xi_{v_l}, \quad \text{for } \eta=1,\dots, m.
\end{align*}

 The symbol $\sigma_{\eta,j}=\sigma_{\eta}(\xi_1,\dots, \xi_{j-1}, \xi_{j+1},\dots, \xi_m)$ denotes  the $\eta$th elementary symmetric polynomial evaluated  in $(\xi_1,\dots, \xi_{j-1}, \xi_{j+1},\dots, \xi_m)$. It is straightforward to see that $\sigma_{\eta,j}= \sigma_{\eta}-\xi_j \sigma_{\eta-1,j}$ for $\eta=1,\dots, m-1$, and iteratively applying this equality $\eta$ times, we have
$$
\sigma_{\eta,j}=\sum_{l=0}^{\eta} \left( -\xi_{j}\right)^{l} \; \sigma_{\eta-l}.$$
For simplicity of notation,  we write     $\varrho_{n+k.j}=S^+_{n+k,2}(\xi_{j})\;\lambda_{n+k,j}$. Hence,  for $i=1,\dots, m$,
$$
  \sum_{j=1}^{m} \varrho_{n+k.j}\;\sigma_{i,j}=   \sum_{j=1}^{m} \varrho_{n+k.j}\;\left( \sum_{l=0}^{i} \left( -\xi_{j}\right)^{l} \; \sigma_{i-l}\right)=  \sum_{l=0}^{i} \sigma_{i-l}\; \left(\sum_{j=1}^{m}\; \varrho_{n+k.j}  \;\left( -\xi_{j}\right)^{l}\right).
$$
From Lemma \ref{Th-GaussJacobiType-QF} we have for $l\leq 2(d-N)$
$$\int_{-1}^{1}  (-x)^l\, d\mu_{\rho,n}(x)=\sum_{j=1}^{n+k-N} \varrho_{n+k.j}\;(-\xi_j)^l=\sum_{j=1}^{m} \varrho_{n+k.j}\;(-\xi_j)^l+\sum_{j=m+1}^{n+k-N} \varrho_{n+k.j}\;(-\xi_j)^l.$$
Thus, from \eqref{Lem-Acota-lambda-3}
$$
\left|\sum_{j=1}^{m} \varrho_{n+k.j}\;(-\xi_j)^l \right|  \leq     \left|\sum_{j=m+1}^{n+k-N} \varrho_{n+k.j}\;(-\xi_j)^l \right| +   \left| \int_{-1}^{1}  (-x)^l\, d\mu_{\rho,n}(x) \right| \leq \frac{\mathfrak{m}+2}{ \mathfrak{m}}\;\mu_{\rho,n}([-1,1]).
$$
As $\{\xi_1,\dots, \xi_m\}\subset [-1,1]$, it is straightforward to see that $|\sigma_{\eta}|\leq m$ for all $\eta=0,\dots, m$. Therefore, for $i=1,\dots, m$
\begin{equation}\label{Lem-Acota-lambda-4}
\left|\sum_{j=1}^{m} \varrho_{n+k.j}\;\sigma_{i,j}\right|  \leq \sum_{l=0}^{i} \left|\sigma_{i-l}\right|\; \left|\sum_{j=1}^{m}\; \varrho_{n+k.j}  \;\left( -\xi_{j}\right)^{l}\right| \leq\frac{m^2\;(\mathfrak{m}+2)}{ \mathfrak{m}}\;\mu_{\rho,n}([-1,1]).
\end{equation}
Using the previous notation, we write
$$G_{2}(z)=\sum_{j=1}^{m}  \frac{\varrho_{n+k.j}}{(z-\xi_{j})} = \frac{L_{m-1}(z)}{\prod_{j=1}^{m}  (z-\xi_{j})} \;  \text{ where }  L_{m-1}(z)=\sum_{j=1}^{m} \varrho_{n+k.j}\; \prod_{\overset{i=1}{i\neq j}}^{m}  (z-\xi_{i}).$$
From the classical Formula  of Vi\`{e}te, $\dsty  \prod_{\overset{i=1}{i\neq j}}^{m}  (z-\xi_{i})=    \sum_{i=0}^{m-1} (-1)^{i}\; \sigma_{i,j}\;  z^{m-1-i}$ (see  \cite[(1.2.2)]{RahSch02}) and
$$
  L_{m-1}(z)=   \sum_{j=1}^{m} \varrho_{n+k.j}\; \left(\sum_{i=0}^{m-1} (-1)^{i}\; \sigma_{i,j}\;  z^{m-1-i}\right)= \sum_{i=0}^{m-1} (-1)^{i}\; \left(\sum_{j=1}^{m} \varrho_{n+k.j}\;\sigma_{i,j}\right)\;  z^{m-1-i}.
$$
Let $\dsty \mathfrak{M}=\max_{z\in K}|z|$. According to \eqref{Lem-Acota-lambda-4}, for all $z \in K$,
\begin{align}\nonumber
\left|  L_{m-1}(z)\right|\leq  & \sum_{i=0}^{m-1} \left|\sum_{j=1}^{m} \varrho_{n+k.j}\;\sigma_{i,j}\right|\; | z|^{m-1-i}\leq\frac{m^2\;(\mathfrak{m}+2)\;\mu_{\rho,n}([-1,1])}{ \mathfrak{m}}  \sum_{i=0}^{m-1}  | z|^{m-1-i} , \\ \nonumber
  \leq & \frac{\mathfrak{m}+2}{\mathfrak{m}}\;m^3\;\max\{\mathfrak{M}^{m-1},1\} \;\mu_{\rho,n}([-1,1])\leq \frac{\mathfrak{m}+2}{\mathfrak{m}}\;m^3\;\max\{\mathfrak{M}^{m-1},1\} \;\mathfrak{M}_{\rho}= \mathfrak{M}_1.\\ \label{Lem-Acota-lambda-5}
 \left| G_{2}(z)\right|=&  \frac{\left|  L_{m-1}(z)\right|}{\prod_{j=1}^{m}  \left| z-\xi_{j}\right|} \leq \frac{\mathfrak{M}_1}{\mathfrak{m}^m}.
\end{align}
Finally, \eqref{Lem-Acota-lambda-3} and \eqref{Lem-Acota-lambda-5} establish the  second assertion.
\end{proof}

\section{Proof of Theorem \ref{Ext_Markov_Th}}\label{Sect-ProofTh2}

Denote $ R_n^{[k]}=\frac{S^{[k]}_{n}(z)}{S_{n+k}(z)}$ and let $\widehat{\mu}_{k}(z)= \int_{-1}^{1} \frac{Q_{k-1}(x)}{z-x} d\mu_{\rho}(x)$ be the  $k$th Markov-type function  associated to $\mu_{\rho}$ ($k\in \NN$) as in \eqref{Th.Markov}. Note that  ${\widehat{\mu}_{k}}(z)$ is well defined and  holomorphic in $\Omega_{\infty}$ ( ${\widehat{\mu}_{k}}\in \HH(\Omega_{\infty})$ for short)  and $\widehat{\mu}_{k}(\infty)=0$.

For the remainder $\left( \widehat{\mu}_{k}(z) - R^{[k]}_n(z)\right)$, the following formulas take place.

\begin{lemma} \label{LemmRepIntegral} Let $\mu$ be a positive Borel measure supported on $[-1,1]$ and $S_n(z)$ and $S^{[k]}_{n}(z)$ defined as above. Then,
\begin{align}\label{LemmRepIntegral-3}
\widehat{\mu}_{k,n}(z) - R^{[k]}_{n,1}(z)=S^+_{n+k,2}(z) \left(\widehat{\mu}_{k}(z) - R^{[k]}_n(z)\right)={\mathcal{O}}\left( \frac{1}{z^{2(n+1)+k-d-N}}\right),
\end{align}
where $\dsty  \widehat{\mu}_{k,n}(z) =  \int_{-1}^{1} \frac{ Q_{k-1}(x)}{z-x} d\mu_{\rho,n}(x).$
 \end{lemma}

\begin{proof}
 From the definition of $S^{[k]}_{n}$, we get
\begin{align*}
  S^{[k]}_{n}(z) =&  \int_{-1}^{1} \frac{S_{n+k}(z)-S_{n+k}(x)}{z-x} Q_{k-1}(x) d\mu_{\rho}(x)\\
  = & S_{n+k}(z)  \int_{-1}^{1} \frac{Q_{k-1}(x)}{z-x} d\mu_{\rho}(x)-
  \int_{-1}^{1}\frac{S_{n+k}(x)\; Q_{k-1}(x)}{z-x} d\mu_{\rho}(x)\\
 =& S_{n+k}(z)\widehat{\mu}_{k}(z)-    \int_{-1}^{1}\frac{S_{n+k}(x)\; Q_{k-1}(x)}{z-x} d\mu_{\rho}(x).
\end{align*}
Then, we have
\begin{align}\label{LemmRepIntegral-1}
\widehat{\mu}_{k}(z) - R^{[k]}_n(z)= \int_{-1}^{1}\frac{S_{n+k}(x)\; Q_{k-1}(x)}{S_{n+k}(z)\;(z-x)} d\mu_{\rho}(x).
\end{align}
On the other hand, from the orthogonality condition \eqref{Sobolev-Orth}
\begin{align*}
  0= & \left\langle S_{n+k}(x), \;\frac{(S_{n-d+1}(z) -S_{n-d+1}(x))\; Q_{k-1}(x)\rho(x)}{z-x}\right\rangle \\
  = & \int_{-1}^{1}  S_{n+k}(x) \;\frac{S_{n-d+1}(z) -S_{n-d+1}(x)}{z-x} \, \; Q_{k-1}(x)d\mu_{\rho}(x).
\end{align*}
Hence, it follows that
\begin{align*}
      \int_{-1}^{1}   \frac{S_{n+k}(x)S_{n-d+1}(z) }{z-x} \; Q_{k-1}(x)\, d\mu_{\rho}(x) =&     \int_{-1}^{1}   \frac{S_{n+k}(x) S_{n-d+1}(x)}{z-x} \; Q_{k-1}(x)\,    d\mu_{\rho}(x),
\end{align*}
and from \eqref{LemmRepIntegral-1}, we obtain
\begin{align*}
\widehat{\mu}_{k}(z) - R^{[k]}_n(z)= & \int_{-1}^{1}   \frac{S_{n+k}(x)}{S_{n+k}(z)}\frac{Q_{k-1}{(x)}}{z-x} \; d\mu_{\rho}(x) = \int_{-1}^{1}   \frac{S_{n+k}(x) S_{n-d+1}(x)}{S_{n+k}(z)S_{n-d+1}(z)}\frac{Q_{k-1}{(x)}}{z-x} \;  d\mu_{\rho}(x)\\
 = & {\mathcal{O}}\left( \frac{1}{z^{2(n+1)+k-d}}\right)
\end{align*}
The second equality in \eqref{LemmRepIntegral-3} is a direct consequence of the above equality. Lastly, we compute
\begin{align}\nonumber
 S^+_{n+k,2}(z)\widehat{\mu}_{k}(z) = &   \int_{-1}^{1} \frac{S^+_{n+k,2}(z)\; Q_{k-1}(x)}{z-x} d\mu_{\rho}(x)\\ \label{LemmRepIntegral-31}
 =  & \int_{-1}^{1} \frac{S^+_{n+k,2}(z)-S^+_{n+k,2}(x)}{z-x} \; Q_{k-1}(x) d\mu_{\rho}(x) + \widehat{\mu}_{k,n}(z),
\end{align}
and
\begin{align}  \label{LemmRepIntegral-32}
\begin{split}
 S^+_{n+k,2}(z) R^{[k]}_n(z) = &   \frac{{(-1)^{\nu}}S^{[k]}_{n}(z)}{S_{n+k,1}(z)} = \frac{{(-1)^{\nu}}}{S_{n+k,1}(z)} \int_{-1}^{1} \frac{S_{n+k}(z)-S_{n+k}(x)}{z-x}\; Q_{k-1}(x)  \;d\mu_{\rho}(x),\\
= &\int_{-1}^{1} \frac{S_{n+k,1}(z)S^+_{n+k,2}(z)-S_{n+k,1}(x)S^+_{n+k,2}(x)}{S_{n+k,1}(z) (z-x)}\; Q_{k-1}(x)  \;d\mu_{\rho}(x),\\
= &R^{[k]}_{n,1}(z) +\int_{-1}^{1} \frac{S^+_{n+k,2}(z)-S^+_{n+k,2}(x)}{z-x}\; Q_{k-1}(x)  \;d\mu_{\rho}(x).
\end{split}
\end{align}
The first equality now follows by subtracting \eqref{LemmRepIntegral-32} from \eqref{LemmRepIntegral-31}.
\end{proof}

\begin{proof}[Proof of Theorem \ref{Ext_Markov_Th}] \
Let $K$ be any compact set on $\Omega_{\infty}$ and consider the level curve $\ell_{\tau}$ defined by
$$\ell_{\tau}=\{z\in \CC : |\varphi(z)|=\tau \}, \quad \text{ where } \tau>1 \text{ and $\varphi$ as in \eqref{AympRatioStandard}.}$$
Since $\varphi(K)$ is a compact set, we can take $\tau$ sufficiently close to $1$ such that $1<\tau<\min|\varphi(K)|$  (remember that $\varphi$ is the conformal map of  the exterior of $[-1,1]$ onto the exterior of the unit circle).
{From Lemma \ref{Lem-Acota-lambda} and \eqref{MesureBounded}, the sequences  $\{ \widehat{\mu}_{k,n}\}$ and $\{R^{[k]}_{n,1}\}$  are uniformly bounded over $\ell_{\tau}$}. Then, there exists a constant $\mathfrak{L}_{\tau}$, independent of $n$, such that  for all  $z \in \ell_{\tau}$
\begin{equation}\label{Ext_Markov_Th-1}
  \left|\left( \widehat{\mu}_{k,n}(z) - R^{[k]}_{n,1}(z)\right)\varphi^{2(n+1)+k-d-N}(z)\right|\leq \mathfrak{L}_{\tau} \; \tau^{2(n+1)+k-d-N}.
\end{equation}
Taking  into account that $\varphi$  has a simple pole at $\infty$, from  \eqref{LemmRepIntegral-3}, we have
$$\left(\left( \widehat{\mu}_{k,n}- R^{[k]}_{n,1}\right)\varphi^{2(n+1)+k-d-N}\right) \in \HH(\Omega_{\infty}).$$
Now,  from the maximum modulus principle  the bound \eqref{Ext_Markov_Th-1} also holds on $K$. Consequently, we have
$$
\left| \widehat{\mu}_{k,n}(z) - R^{[k]}_{n,1}(z) \right|\leq \mathfrak{L}_{\tau} \; \left(\frac{\tau}{\left|\varphi(z) \right|}\right)^{2(n+1)+k-d-N}  \quad z \in K.
$$ {Hence}
\begin{equation}\label{Ext_Markov_Th-2}
  \sup_{z\in K}{\left|\widehat{\mu}_{k,n}(z) - R^{[k]}_{n,1}(z)\right|} \leq \mathfrak{L}_{\tau} \; \left(\frac{\tau}{{\min|\varphi(K)|}}\right)^{2(n+1)+k-d-N}  \limn  0,
\end{equation}
which {is equivalent} to say that $\dsty  R^{[k]}_{n,1}(z) \unifn \widehat{\mu}_{k,n}(z) \quad K \subset \Omega_{\infty}.$

As before, ${\Omega_{\infty}^*=\Omega_{\infty}} \setminus \{c_1,c_2,\dots, c_N\}$. For the rest of the proof we  assume that the compact set $K$ is a subset of $\Omega_{\infty}^*$.  From Lemma \ref{AsymCBounded_Lemma}, there exists a constant $\mathfrak{L}_2>0$, independent of $n$, such that $\dsty \mathfrak{L}_2  \leq  |S_{n+k,2}(z)| $ for all $z \in K$. Therefore, taking into account \eqref{LemmRepIntegral-3}, we get
\begin{equation}\label{Ext_Markov_Th-3}
 \sup_{z\in K}{\left|\widehat{\mu}_{k}(z) - R^{[k]}_{n}(z)\right|} \leq \frac{\mathfrak{L}_{\tau}}{\mathfrak{L}_2   } \; \left(\frac{\tau}{{\min|\varphi(K)|}}\right)^{2(n+1)+k-d-N}  \limn 0,
\end{equation}
\end{proof}

As a complement of  Theorem \ref{Ext_Markov_Th}, we have the following estimate for the degree of convergence (``speed'') of  $R^{[k]}_n$ to $\widehat{\mu}_{k}$.

\begin{corollary} \label{CoroSpeed}Under the same hypotheses of Theorem \ref{Ext_Markov_Th}, we have
\begin{equation}\label{CoroSpeed-1}
  \limsup_n \left\|\widehat{\mu}_{k} - R^{[k]}_n\right\|_K^{1/2n}\leq \|\varphi^{-1}\|_K<1.
\end{equation}
\end{corollary}

\begin{proof} Taking the $2n$th root in \eqref{Ext_Markov_Th-3}, we get
\begin{equation}\label{CoroSpeed-2} \left\|\widehat{\mu}_{k} - R^{[k]}_n\right\|_K^{1/2n}\leq  \left(\frac{\mathfrak{L}_{\tau}}{\mathfrak{L}_2   }\right)^{1/2n} \; \left(\frac{\tau}{{\min|\varphi(K)|}}\right)^{(2(n+1)+k-d-N)/(2n)}.\end{equation}
Since  ${\tau<\min|\varphi(K)|}$,  \eqref{CoroSpeed-1} follows from \eqref{CoroSpeed-2}.
\end{proof}


\end{document}